\documentclass{article}

\usepackage[english]{babel}
\usepackage{amsmath}
\usepackage{mathtools}
\usepackage{graphicx} 
\usepackage{amsthm}
\usepackage{fancyhdr}
\usepackage{amsfonts}
\usepackage{amssymb}
\usepackage{pdfpages}
\usepackage{musicography}
\usepackage{tikz}
\usepackage{url}
\usepackage{comment}
\usepackage[shortlabels]{enumitem}
\usepackage{comment}

\newtheorem{corollary}{Corollary}

\theoremstyle{definition}
\newtheorem*{remark}{Remark}

\newtheorem{prop}{Proposition}

\newtheorem{definition}{Definition}

\newtheorem{lemma}{Lemma}

\newtheorem{theorem}{Theorem}

\newcommand{\Der}[2][]{\frac{\partial#1}{\partial#2}}
\newcommand{\der}[2][]{\frac{d#1}{d#2}}
\newcommand{\Norm}[1]{\Vert#1\Vert}
\newcommand{\norm}[1]{\vert#1\vert}

\newcommand{\orthcomp}[1]{\langle #1 \rangle^{\bot}}

\newcommand{\wholdersp}[4]{P^{#1,#2,\alpha}_{\mu#3#4}}

\newcommand{\wholdernorm}[5]{\Norm{#1}_{\wholdersp{#2}{#3}{#4}{#5}}}

\title{Lagrangian mean curvature flow in the Kummer K3 surface}
\author{Carlos Alberto Ochoa Flores}

\begin{document}

\maketitle

\begin{abstract}

A major problem in geometric analysis is to understand the behaviour of the Lagrangian mean curvature flow. This has proved to be a challenging problem, in particular, not many explicit examples of flows that exist for all time and converge to a minimal Lagrangian are known. In this paper we construct new examples of Lagrangian mean curvature flows in the Kummer K3 surface that converge to special Lagrangian spheres. We reduce the construction to a scalar perturbation problem that can be solved by a fixed point argument. 

\end{abstract}

\section*{Introduction}

A long-standing question for Calabi--Yau manifolds has been to understand which Hamiltonian isotopy classes contain a special Lagrangian representative, this is of particular interest because of its potential applications in mirror symmetry. Special Lagrangians are minimal submanifolds, thus a natural approach to studying them is to use the mean curvature flow. In fact, the mean curvature flow in K\"ahler--Einstein manifolds preserves the Lagrangian condition \cite{smoczyk1996canonical}, leading to the notion of the Lagrangian mean curvature flow (LMCF). Thomas and Yau \cite{thomas2001special} conjectured the existence of a stability condition that guarantees the long-time existence of the flow and its convergence to a special Lagrangian. However, Neves proved that in Calabi--Yau surfaces, any Lagrangian is in the same Hamiltonian isotopy class as a Lagrangian manifold that develops singularities under the LMCF \cite{neves2013finite}. which provides strong evidence that an understanding of the possible singularities for the flow might be essential in the resolution of the Thomas--Yau conjecture. This motivated Joyce to update the conjecture by suggesting that there might be a connection between the behaviour of the flow and Bridgeland stability conditions on Fukaya categories \cite{joyce2015conjectures}, his proposal involves the construction of a flow with surgeries. 

Even though there has been some progress towards the understanding of the possible singularities, specially for Calabi--Yau surfaces \cite{szkelyhidi2026generic, lotay2022neck}, there is still a lot to be figured out. On the other hand, not many explicit examples of Lagrangian mean curvature flows that converge to a special Lagrangian are known. In certain cases, where symmetries are present, the problem can be reduced to understanding a modified curve shortening flow. This idea was used by Lotay and Oliveira \cite{lotay2020special} to understand the behaviour of the LMCF for circle invariant Lagrangians in ALE and ALF manifolds that arise from the Gibbons--Hawkins ansatz. They were able to prove a version of the Thomas--Yau conjecture in this case. In particular, they were able to construct several examples of LMCF in the Eguchi--Hanson space that converge smoothly to a special Lagrangian. 

A Calabi--Yau structure on a smooth $2n$-dimensional manifold $M$ is a triple $(J,\omega,\Omega)$, where $J$ is an integrable complex structure on $M$, $\omega$ is a K\"ahler form on $M$ that is compatible with $J$, and $\Omega$ is a nowhere vanishing holomorphic $n-$ form, that satisfies 
\begin{align*}
    \frac{\omega^n}{n!}=(-1)^{\frac{n(n-1)}{2}}\bigg(\frac{i}{2}\bigg)^n\Omega\wedge \overline{\Omega}. 
\end{align*}
Notice that these three objects determine a complex, a symplectic and a Riemannian structure in $M$.

If we restrict the holomorphic form of a Calabi--Yau manifold to a Lagrangian submanifold, then it is possible to see that it has unit norm. This allows us to prove the following:

\begin{prop}[See Proposition 2.4.5 in \cite{wood2020singularities}]
    Let $f:L\rightarrow (M, J,\omega,\Omega)$ be an oriented Lagrangian submanifold of a Calabi--Yau manifold. Then there exists a smooth function $\theta(f): L\rightarrow S^1$ such that
    \begin{align*}
        \Omega\vert_L=e^{i\theta(f)}vol_L.
    \end{align*}
\end{prop}

We call $\theta$ the Lagrangian angle of $L$. Notice that the differential of $\theta$ is a well defined $1$-form on $L$. We say that the Lagrangian $f:L\rightarrow (M,J,\omega,\Omega)$ is zero-Maslov if its Lagrangian angle can be lifted to a real valued function. In this case we will also denote said function by $\theta(f)$. 

We say that the Lagrangian $f:L\rightarrow(M,J,\omega,\Omega)$ is a Special Lagrangian if 
\begin{align*}
    f^*Im\Omega=0.
\end{align*}
Notice that in this case we can choose $\theta(f)$ to be constant. Moreover, in this case $Re\Omega$ is a calibration for $f$; in particular $f$ is a minimal Lagrangian.  

\begin{remark}
    From now on, we will refer to a smooth, closed, oriented Lagrangian simply as a Lagrangian. 
\end{remark}

Motivated by the fact that the Eguchi--Hanson space models the curvature concentration region of the Kummer K3 surface, Lotay and Oliveira conjectured that one should be able to perturb the examples they constructed in \cite{lotay2020special} to obtain LMCF's in the Kummer K3 surface that exist for all time, and converge to a special Lagrangian. The main objective of this paper is to show that this can indeed be done when the original flow is graphical over the special Lagrangian (see Definition \ref{graphicalLMCF}). Informally, our main result says the following (see Theorem \ref{Maintheorem1Tech} for a detailed statement).

\begin{theorem}\label{Maintheo1Informal}
    A Lagrangian mean curvature flow in the Eguchi--Hanson space, that is graphical over the zero section and that converges smoothly to it, can be perturbed to obtain a flow in the Kummer K3 surface that converges smoothly to one of the sixteen special Lagrangian spheres in the curvature concentration region of the surface. 
\end{theorem}

The theorem follows easily from an abstract deformation result for Lagrangian mean curvature flows along a smooth family of Calabi--Yau structures (see Theorem \ref{deformationtheo} for a detailed statement): 

\begin{theorem}
    Consider a smooth family of Calabi--Yau structures $(M,\omega,\Omega_{\varepsilon})$, $\varepsilon\in [0,\varepsilon_0)$. Suppose that $L_{\infty}$ is a special Lagrangian with respect to any member of the family. For small enough $\varepsilon$, an LMCF $L_t$ in $(M,\omega,\Omega_0)$, that is graphical over $L_{\infty}$ and that converges smoothly to it, can be perturbed to an LMCF $\tilde{L}_{t}$ in $(M,\omega,\Omega_{\varepsilon})$ that also converges to $L_{\infty}$. 
\end{theorem}

Given the generality of the theorem, it should be possible to use it in other contexts. In particular, it may be used to construct examples of LMCF's in Calabi--Yau manifolds that arise from gluing constructions. 

The main bulk of the paper is dedicated to proving the abstract deformation result. Its proof can be summarized as follows:
\begin{enumerate}
    \item Suppose $F_0$ is an LMCF in $(M,\omega,\Omega_0)$. First we set up the problem as a perturbation problem by assuming that the new flow in $(M,\omega,\Omega_{\varepsilon})$ will be of the form
    \begin{align*}
        F_{\varepsilon}=F_0+E_{\varepsilon},
    \end{align*}
    then we solve for $E_{\varepsilon}$. In a Calabi--Yau manifold, the mean curvature vector of a Lagrangian is given by the gradient of its Lagrangian angle. Since we are assuming that our initial flow is graphical, we can see that $E_{\varepsilon}$ solves an equation that is equivalent to a scalar equation of the form
    \begin{align*}
        \der[u]{t}= P(u), 
    \end{align*}
    where $P$ is a fully nonlinear second order elliptic operator. 
    \item  By linearising at $u=0$, we see that the equation can be written as 
    \begin{align*}
        \der[u]{t}-\mathcal{L}(u)= C_{\varepsilon}+Q_{\varepsilon}(u).
    \end{align*}
    Here $\mathcal{L}$ is a second order linear elliptic operator, $C$ represents the zero order terms of the PDE, and $Q$ represents its quadratic and higher order terms. 
    \item We study the linear operator. In order to do this, we introduce weighted H\"older spaces, with good compactness properties, and where the linear operator has a nice existence theory.  
    \item Finally, we prove estimates on the zero order and higher order terms of the equation, which allow us to set up an iteration map, that satisfies the hypothesis of the Schauder fixed point theorem. This essentially boils down to studying the behaviour of the Lagrangian angle. 
\end{enumerate}

The paper is organized as follows: In the first section we recall a few facts about LMCF that we will use throughout the paper. In the second section, we carry out the proof of the abstract deformation result. Finally, in the third section we briefly describe the construction of the Kummer K3 surface, and we show that our abstract deformation result can indeed be used to prove Theorem \ref{Maintheo1Informal}. 

\section{Preliminaries}

Recall that each submanifold $f:L\rightarrow (M,g)$ of a Riemannian manifold has a normal vector field, called the mean curvature vector, associated to it. We will denote it by $\overrightarrow{H}(f)$. We say that a time dependent family of submanifolds, $F:L\times[t_1,t_2)\rightarrow M$, is a mean curvature flow if the mean curvature vector satisfies 
\begin{align}\label{MCF}
    \Der[]{t}F^{\bot}=\overrightarrow{H}.
\end{align} 

As mentioned in the introduction, by the work of Smoczyk \cite{smoczyk1996canonical}, it is known that a mean curvature flow in a Calabi--Yau manifold that starts at a Lagrangian submanifold preserves the Lagrangian condition. Thus it makes sense to talk about the Lagrangian Mean Curvature Flow (LMCF) in a Calabi--Yau manifold. 

In our applications, it will be important to know that if a LMCF converges to a special Lagranian then it does so exponentially fast in every $C^k$-norm. 

\begin{prop}[See \cite{li2009convergence}] \label{expconvergence}
    Let $L$ be a closed smooth manifold. Let $F: L\times [0,\infty)\rightarrow (M, \omega, \Omega)$ be an LMCF, in a Calabi--Yau manifold, that converges to a Special Lagrangian $f_{\infty}: L\rightarrow (M,\omega,\Omega)$. Then for each $k\geq 0$ there are constants $C_k>0$, and $\lambda_k>0$ such that 
    \begin{align*}
        \Norm{\overrightarrow{H}(F(\cdot,t))}_{C^k}< C_ke^{-\lambda_kt}, \quad t\in [0,\infty).
    \end{align*}
    Moreover, $F(\cdot,t)$ converges exponentially fast to $f_{\infty}$. 
\end{prop}

The Lagrangian neighbourhood theorem guarantees that every compact Lagrangian submanifold $L\subset (M,\omega)$ has a neighbourhood that is symplectomorphic to a neighbourhood of the zero section in $T^*L$.

\begin{theorem}[See Theorem 3.4.13 in \cite{mcduff2017introduction}]
    Let $(M,\omega)$ be a symplectic manifold and $L\subset M$ be a compact Lagrangian submanifold. Then there exists a neighbourhood $U\subset T^*L$ of the zero section, a neighbourhood $V\subset M$ of $L$ and a diffeomorphism $\phi:U\rightarrow V$ such that
    \begin{align*}
        \phi^*\omega=\omega_{can}, \quad \phi\vert_L=\text{id}.
    \end{align*}
    Here, $\omega_{can}$ is the canonical symplectic form in $T^*L. $
\end{theorem}

As a consequence, in certain cases it is possible to rewrite the LMCF equation as a PDE for differential forms. In our context, this will allow us to reduce the problem to a scalar PDE. 

\begin{definition}\label{graphicalLMCF}
Consider a Lagrangian submanifold, 
\begin{align*}
f_{\infty}:L\rightarrow (M,J,\omega,\Omega).
\end{align*}
Fix a Lagrangian neighbourhood $\phi: U\subset T^*L\rightarrow V\subset M$ of it. We say that a submanifold $f:L\rightarrow (M, \omega, \Omega)$ is graphical over $f_{\infty}$, if there exists a $1$-form $\gamma: L\rightarrow T^*L$, such that $f=\phi\circ \gamma$. 

Similarly, we say that a family of submanifolds $F:L\times [0,T)\rightarrow (M,\omega,\Omega)$ is graphical if there is a time-dependent $1$-form $\gamma$, such that $F=\phi\circ\gamma$. 
\end{definition}

Notice that a graphical submanifold is a Lagrangian if and only if it is represented by a closed form.

Suppose that $F:L\times[0,T)\rightarrow (M,\omega,\Omega)$ is an LMCF that is graphical over $f_{\infty}:L\rightarrow (M,\omega,\Omega)$, then
\begin{align}\label{originalLMCF}
    \Der[F]{t}^{\bot}=\overrightarrow{H}, \quad F^*\omega=0,
\end{align}
and there is a Lagrangian neighbourhood $\phi:U\rightarrow V$, and a time-dependent $1$-form $\gamma$ such that $F=\phi\circ\gamma$. 

As we mentioned, the Lagrangian condition is equivalent to $d\gamma=0$. On the other hand, notice that the Lagrangian condition implies that 
\begin{align}
    F^*\omega(\Der[F]{t},\cdot)=F^*\omega(\Der[F]{t}^{\bot},\cdot)=F^*\omega(\overrightarrow{H}, \cdot)=H,
\end{align}
for a $1$-form $H$, known as the mean curvature $1$-form. It is well known that, for Lagrangians, it coincides with the differential of the Lagrangian angle. 

\begin{prop}[See Proposition 2.4.6 in (\cite{wood2020singularities})]\label{lagangle}
    Let $F:L\rightarrow M$ be an oriented Lagrangian submanifold of a Calabi--Yau manifold. Let $\theta$ be its Lagrangian angle, and let $H$ be its mean curvature $1$-form. Then $d\theta=H$. 
\end{prop}

Given that $\phi^*\omega=\omega_{can}$, we can prove the following

\begin{lemma}[See Lemma 4.1 in \cite{su2024infinite}]\label{pullback}
    Let $f:L\rightarrow M$ be a Lagrangian submanifold and $\phi:U\rightarrow M$ be a Lagrangian neighborhood. Then, given a time-dependent family of closed $1$-forms $\gamma$ on $L$ whose image belongs to $U$,
    \begin{align*}
        (\phi\circ\gamma)^*(\omega(\Der[(\phi\circ\gamma)]{t},\cdot)=-\der[\gamma]{t}. 
    \end{align*}
\end{lemma}

It follows from Lemma \ref{pullback}, and Proposition \ref{lagangle} that

\begin{prop}\label{IntegratedLMCF}
Let $f_{\infty}:L\rightarrow (M,J,\omega,\Omega)$ be a Lagrangian. Fix a Lagrangian neighbourhood $\phi: U\subset T^*L\rightarrow V\subset M$ of it. Then $\phi\circ\gamma: L\times [0, T)\rightarrow (M,\omega,\Omega)$ is a graphical LMCF if and only if 
\begin{align}\label{originalformLMCF}
    \der[\gamma]{t}=d\theta(\gamma), \quad d\gamma=0.
\end{align}
\end{prop}

\section{Abstract deformation result}

Consider a smooth family of Calabi--Yau structures $(M,\omega,\Omega_{\varepsilon})$, $\varepsilon\in[0,\varepsilon_0)$. Notice that the complex and metric structures are allowed to change along the family, but the symplectic structure is fixed. 

 The main purpose of this section is to describe how to perturb an LMCF in the Calabi--Yau manifold $(M,\omega,\Omega_0)$ to an LMCF in the Calabi--Yau $(M,\omega,\Omega_{\varepsilon})$ for $\varepsilon>0$ small enough. In particular, we prove the following result:

\begin{theorem}\label{deformationtheo}
     Consider a zero Maslov LMCF $F_0:L\times [0,\infty)\rightarrow (M,\omega,\Omega_{0})$ that converges to a special Lagrangian $f_{\infty}:L\rightarrow (M,\omega,\Omega_0)$. Moreover, suppose that $F_0$ is graphical over $f_{\infty}:L\rightarrow M$, and that $f_{\infty}:L\rightarrow (M, \omega, \Omega_{\varepsilon})$ is a special Lagrangian. Then, for small enough $\varepsilon$, the LMCF $F_{\varepsilon}:L\times [0,T)\rightarrow (M,\omega,\Omega_{\varepsilon})$, with the same initial condition as $F_0$, exists for all time and converges to $f_{\infty}: L\rightarrow M$.  
\end{theorem}

We prove the theorem by a fixed point argument. To do this there are three key steps. First we use Proposition \ref{IntegratedLMCF} to interpret the problem as a scalar PDE. Afterwards, we write the PDE as a ``quasilinear" problem in an adequate function space where the linearization is well behaved. Finally we prove estimates on the zeroth and nonlinear terms of the PDE, that allow us to define an iteration mapping that satisfies the hypothesis of the Schauder fixed point theorem. 

\subsection{Integrating the perturbation problem}

We shall use the notation of Theorem \ref{deformationtheo}. We denote the Lagrangian angle with respect to $(M,\omega,\Omega_{\varepsilon})$ by $\theta_{\varepsilon}$. 

Consider a Lagrangian neighbourhood $\phi:U\subset T^*L\rightarrow M$. Suppose $F_0=\phi\circ \gamma$, for a time-dependent $1$-form $\gamma$. It follows from Proposition \ref{IntegratedLMCF} that
\begin{align*}
    \der[\gamma]{t}=d\theta_0(\gamma), \quad d\gamma=0.
\end{align*}

Similarly, suppose that the perturbed flow will be of the form
\begin{align*}
F_{\varepsilon}=\phi\circ(\gamma+du),
\end{align*}
and denote its mean curvature vector by $\overrightarrow{H_{\varepsilon}}$. Then, 
\begin{align}
    \Der[F_{\varepsilon}]{t}^{\bot}=\overrightarrow{H}_{\varepsilon}, \quad F_{\varepsilon}^*\omega=0, 
\end{align}
is equivalent to 
\begin{align}
    \der[(\gamma+du)]{t}=d\theta_{\varepsilon}(\gamma+du), \quad d(\gamma+du)=0.
\end{align}
We can conclude that $F_{\varepsilon}$ is a LMCF if and only if $u$ satisfies
\begin{align}\label{perturbedproblem}
    \der[u]{t}=\theta_{\varepsilon}(\gamma+du)-\theta_0(\gamma)+a(t),
\end{align}
where $a(t)$ is a result of integrating $1$-forms into functions, which can be done because we are assuming that $\gamma$ is zero Maslov. 

Notice that we want $F_\varepsilon$ and $F_0$ to have the same initial condition. In conclusion, we want to solve the problem

\begin{align}
    \begin{cases}
        \der[u]{t}=\theta_{\varepsilon}(\gamma+du)-\theta_0(\gamma)+a(t),\\
        u(\cdot,0)=0.
    \end{cases}
\end{align}

\subsection{Linear theory}

In this section we explain how to compute the linearization of (\ref{perturbedproblem}) at $0$. Afterwards we develop the necessary linear theory to be able to set up the fixed point problem. 

First of all, it is convenient to have an expression for the deformation vector field of our original solution $\gamma$ in the direction of $du$. Such an expression was computed by Behrndt in \cite{behrndt2011generalized}.
\begin{lemma} [See Lemma 4.11 in \cite{behrndt2011generalized}]
    The deformation vector field of $F_{0}=\phi\circ \gamma$ in the direction of $du$ is given by 
    \begin{align}\label{deformationfield}
        \der[]{s}\vert_{s=0}\phi\circ(\gamma+ d(su))= J(F_0)_*\nabla_{\gamma_{\varepsilon}} u+\hat{V}(du),
    \end{align}
    where $\hat{V}(du)$ is the tangential component of the variation field, and $\nabla_{\gamma_{\varepsilon}}$ is the gradient with respect to $\gamma^*g_{\varepsilon}$.
\end{lemma}

With this in mind, we can compute the linearisation of $\theta_{\varepsilon}(\gamma+du)-\theta_0(\gamma)$ at $u=0$ with the same arguments one uses to prove that the Lagrangian angle of an LMCF evolves by the heat equation (see section 3.2 in \cite{wood2020singularities}, and section 5 in \cite{su2024infinite}). 

\begin{prop}
    The linearisation of $\theta_{\varepsilon}(\gamma+du)-\theta_0(\gamma)$ at $0$ is given by 
    \begin{align}
        \der[]{s}\vert_{s=0}\theta_{\varepsilon}(\gamma+d(su))=\Delta_{\gamma_{\varepsilon}}u+\langle \nabla_{\gamma_{\varepsilon}}\theta_{\varepsilon}(\gamma), \hat{V}(du)\rangle_{\gamma_{\varepsilon}}.
    \end{align}
\end{prop}
\begin{proof}
    Suppose $\Phi_s$ denotes the flow associated to (\ref{deformationfield}). If we write $L_s=\phi(\gamma+d(su))$, and $\Omega_s=\Omega_{\varepsilon}\vert_{L_s}$ then
    \begin{align*}
        \der[\Omega_s]{s}\vert_{s=0}&=\der[]{s}(\Phi_0^*\Phi_s^*\Omega_{\varepsilon})\vert_{s=0}\\
        &=\Phi_0^*\der[]{s}\Phi_s^*\Omega_{\varepsilon}\vert_{s=0}\\
        &=\Phi_0^*(\mathcal{L}_{J\nabla_{\gamma_{\varepsilon}} u+\hat{V}(du)}\Omega_{\varepsilon})\\
        &=\Phi_0^*(d(\iota_{J\nabla_{\gamma_{\varepsilon}} u+\hat{V}(du)}\Omega_{\varepsilon}))\\
        &=d(ie^{i\theta_{\varepsilon}(\gamma)}\iota_{\nabla_{\gamma_{\varepsilon}} u}Vol_{\gamma_{\varepsilon}}+e^{i\theta_{\varepsilon}(\gamma)}\iota_{\hat{V}(du)}Vol_{\gamma_{\varepsilon}})\\
        &=ie^{i\theta_{\varepsilon}(\gamma)}d*_{\gamma_{\varepsilon}}du+ie^{i\theta_{\varepsilon}(\gamma)}\langle d\theta,\hat{V}(du)\rangle_{\gamma_{\varepsilon}} Vol_{\gamma_{\varepsilon}}+e^{i\theta_{\varepsilon}(\gamma)}(real\ terms)\\
    \end{align*}
    On the other hand 
    \begin{align*}
        \der[]{s}\Omega_s\vert_{s=0}= ie^{i\theta_{\varepsilon}(\gamma)}\der[]{s}\theta_{\varepsilon}(\gamma+d(su))\vert_{s=0}Vol_{\gamma_{\varepsilon}}+e^{i\theta_{\varepsilon}(\gamma)}(real\ terms)
    \end{align*}
    Comparing imaginary parts, after multiplying by $e^{i\theta_{\varepsilon}(\gamma)}$, it follows that 
    \begin{align*}
        \der[]{s}\theta_{\varepsilon}(\gamma+d(su))\vert_{s=0}Vol_{\gamma_{\varepsilon}}=d*_{\gamma_{\varepsilon}}du+\langle d\theta,\hat{V}(du)\rangle_{\gamma_{\varepsilon}} Vol_{\gamma_{\varepsilon}}.
    \end{align*}
    By the definition of the Hodge star, it follows that
     \begin{align*}
        \der[]{s}\vert_{s=0}\theta_{\varepsilon}(\gamma+d(su))=\Delta_{\gamma_{\varepsilon}}u+\langle \nabla_{\gamma_{\varepsilon}}\theta_{\varepsilon}(\gamma), \hat{V}(du)\rangle_{\gamma_{\varepsilon}}.
    \end{align*}
\end{proof}

The previous proposition tells us that we can write (\ref{perturbedproblem}) as a ``quasilinear" PDE, 
\begin{align}
    \der[u]{t}-\Delta_{\gamma_{\varepsilon}}u-\langle\nabla_{\gamma_{\varepsilon}}\theta_{\varepsilon}(\gamma),\hat{V}(du)\rangle_{\gamma_{\varepsilon}}=\theta_{\varepsilon}(\gamma)-\theta_0(\gamma)+Q_{\varepsilon}(u)+a(t),
\end{align}
where
\begin{align*}
    Q_{\varepsilon}(u)= \theta_{\varepsilon}(\gamma+du)-\theta_{\varepsilon}(\gamma)-\Delta_{\gamma_{\varepsilon}}u-\langle\nabla\theta_{\varepsilon}(\gamma),\hat{V}(du)\rangle_{\gamma_{\varepsilon}}.
\end{align*}

To solve this equation via fixed point methods, we need to have a good understanding of the existence theory for the linear operator
\begin{align*}
    \mathcal{L}(u)=\der[u]{t}-\Delta_{\gamma_{\varepsilon}}u-\langle\nabla_{\gamma_{\varepsilon}}\theta_{\varepsilon}(\gamma),\hat{V}(du)\rangle_{\gamma_{\varepsilon}}.
\end{align*}

Let us first define appropriate function spaces. In this case, since we are working with functions defined up to infinity in time, we will need to add a weight in time, so that our spaces have adequate compactness properties. 

Let $g_{\infty}$ be the Riemannian metric induced on $L$ by the Calabi--Yau structure at $\varepsilon=0$, when thinking of it as the immersion $f_{\infty}:L\rightarrow M$. Consider $\mu>1$, and $\alpha\in (0,\frac{1}{2})$. Given a time dependent tensor over $L$, consider the following norm and seminorms

\begin{align}
    \Norm{T}_{\mu}:=\sup_{(x,t)\in L\times [0,\infty)} t^{\mu}\norm{T}_{g_{\infty}},
\end{align}

\begin{align}
    [T]_{\mu,\alpha}:=\sup_{t\in [0,\infty)}\sup_{\substack{x_1,x_2\in L\\ d_{g_{\infty}}(x_1,x_2)<\min\{inj(g_{\infty}),1\}}}t^{\mu}\frac{\norm{T(x_1,t)-T(x_2,t)}_{g_{\infty}}}{d_{g_{\infty}}(x_1,x_2)^{2\alpha}},
\end{align}

\begin{align}
    \langle T\rangle_{\mu,\alpha}=\sup_{x\in L}\sup_{t>0}\sup_{\substack{t_1,t_2\in[t,2t]\\ 0<\norm{t_1-t_2}<1}} t^{\mu}\frac{\norm{T(x,t_1)-T(x,t_2)}_{g_{\infty}}}{\norm{t_1-t_2}^{\alpha}}, 
\end{align}

Here the difference $T(x_1,t)-T(x_2,t)$ is understood using the parallel transport along the unique shortest geodesic between $x_1$ and $x_2$ to compare the values. Notice that it is possible to consider this geodesic because $d_{g_{\infty}}(x_1,x_2)<inj(g_{\infty})$.  

Define a weighted parabolic H\"older norm for a tensor $T$ on $L$ by 
\begin{align}
    \Norm{T}_{P^{0,0,\alpha}_{\mu}}:=\Norm{T}_{\mu}+[T]_{\mu,\alpha}+\langle T\rangle_{\mu,\alpha}.
\end{align}
We then define the parabolic H\"older spaces $P^{l,k,\alpha}_{\mu}$ to be the space of functions $u: L\times [0,\infty)\rightarrow \mathbb{R}$ such that the norm
\begin{align}
    \Norm{u}_{P^{l,k,\alpha}_{\mu}}:=\sum_{\substack{0<i\leq k,0<j\leq l\\ 2i+j\leq l}}\Norm{\partial_t^i\nabla^j u}_{P^{0,0,\alpha}_{\mu}}
\end{align}
is finite. We denote the ball of radius $\delta>0$ centered at $0$ in $P^{l,k,\alpha}_{\mu}$ by $B_{\mu,\delta}^{l,k,\alpha}$.

Similarly we can define norms $\wholdernorm{\cdot}{l}{k}{}{}$, and  H\"older spaces $\wholdersp{l}{k}{}{}$ for differential forms. 

As mentioned earlier, to be able to set up the fixed point method, we will need good compactness properties for the H\"older spaces i.e. 
\begin{prop}[See Lemma 7.7 in \cite{su2024infinite}]\label{Schaudercompactness}
    Let $1<\mu' <\mu$, $0<\alpha'< \alpha<\frac{1}{2}$. The inclusion 
    \begin{align}
        P^{l,k,\alpha}_{\mu}\hookrightarrow P^{l,k,\alpha'}_{\mu'}
    \end{align}
    is compact. 
\end{prop}

The necessary linear theory can be summarized in the following theorem,

\begin{theorem}[Infinite-time existence theory for parabolic problems]\label{exttheory}
    Let $g_t$, $t\geq 0$, be a smooth family of Riemannian metrics in a closed manifold $L$ that converges to $g_{\infty}$. Let $f\in P^{0,0,\alpha}_{\mu}\cap\orthcomp{1}$, and suppose $V_{t}$ is a family of smooth vector fields over $L$, that is uniformly bounded. Consider also a vector field $\hat{V}(du)$ such that $\wholdernorm{\hat{V}(du)}{0}{0}{}{}\leq C\wholdernorm{\nabla u}{0}{0}{}{}$, for some $C>0$ that is independent of $u$. Then the problem 
    \begin{align}\label{linprob}
        \begin{cases}
            \Der[u]{t}-\Delta_{g_t}u-\langle V_t, \hat{V}(du)\rangle_{g_t}=f,\\
            u(\cdot,0)=0,
        \end{cases}
    \end{align}
    has a unique solution, and there is a constant $C>0$, such that $\Vert u\Vert_{P^{1,2,\alpha}_{\mu}}\leq C \Vert f\Vert_{P^{0,0,\alpha}_{\mu}}$.
\end{theorem}
\begin{proof}
    The result follows from standard techniques in linear parabolic theory. The most complicated part is obtaining an a priori estimate, $\wholdernorm{u}{1}{2}{}{}\leq C\wholdernorm{f}{0}{0}{}{}$, for solutions of (\ref{linprob}). Once this has been done, the existence of solutions can be proven using the continuity method (see section 8.3 of \cite{wu2006elliptic}).

    To prove the a priori estimate, one can proceed in three steps: 
    \begin{enumerate}
        \item Using separation of variables, an analogous estimate can be proven for solutions of the heat equation, where the Laplacian is taken with respect to any fixed metric in the family.
        \item Given that the family of metrics $g_t$ converges to $g_{\infty}$, the previous step allows one to obtain an a priori estimate for solutions of the heat equation, when the Laplacian is taken with respect to a time-dependent family of metrics. 
        \item Finally, we can use interpolation inequalities satisfied by our H\"older spaces to obtain the desired a priori estimate from the ones obtained in the previous step. 
    \end{enumerate}
\end{proof}

\subsection{Estimates on the nonlinear terms}

The final ingredient to set up the fixed point method are good estimates on the zero order and higher order terms of the PDE. To do this we need to understand the behaviour of the Lagrangian angle. 

\begin{lemma}\label{lagestimate}
    Consider a Lagrangian neighbourhood $(\phi,U)$. Then given a $1$-form $\gamma\in P^{1,1,\alpha}_{\mu}$, in $\phi$, such that $\wholdernorm{\gamma}{1}{1}{}{}<A$, for $A>0$, then there is a $C>0$, $C:= C(A)$, that satisfies 
    \begin{align}
        \Norm{\sin\theta_{\varepsilon}(\gamma)}_{P^{0,0,\alpha}_{\mu}}\leq C(\sum_{
        j=1}^n\wholdernorm{\gamma}{1}{1}{}{}^{j}).
    \end{align}
    In particular this implies that if $\gamma\in P^{1,1,\alpha}_{\mu}$, then $\theta_{\varepsilon}(\gamma)\in P^{0,0,\alpha}_{\mu}$. 
\end{lemma}
\begin{proof}
    Recall that $\sin \theta_{\varepsilon}(\gamma)=*_{{\gamma}_{\varepsilon}}\gamma^{*}Im\Omega_{\varepsilon}$. We can choose local coordinates $(x^1,\ldots, x^n)$ in $L$, and $(y^1,\ldots, y^{2n})$ in $T^*L$ so that $\gamma=(x^1,\ldots, x^n,\gamma_1,\ldots,\gamma_n)$ and $\Omega_{\varepsilon}=\Omega_{ij}dy^i\wedge dy^j$, then one can see that 
    \begin{align}
       \gamma^*Im\Omega_{\varepsilon}= [c_0(x,\varepsilon,\gamma)+\sum_{j=1}^n P_j(x,\varepsilon,\gamma,D\gamma)]dx^1\wedge\ldots\wedge dx^n,
    \end{align}
    where $c_0(x,\varepsilon,0)=0$, and $P_j$ can be seen as a homogeneous polynomial of degree $j$ on its third variable, with bounded coefficients. It follows that the norm of $P_j(x,\varepsilon,\gamma,D\gamma)$ can be controlled by $\Norm{\gamma}_{P_{\mu}^{1,1,\alpha}}^j$. Moreover, we can use the intermediate and mean value theorems, in order to prove that the norm of $c_0(x,\varepsilon,\gamma)$ can be controlled by $\wholdernorm{\gamma}{1}{1}{}{}$. 
\end{proof}

Using the previous lemma, we can prove that the zero order term lives in the adequate function space, and its smallness follows easily after that. 

\begin{prop}[Zero order estimate] \label{zeroorderest}If $\gamma$ represents a graphical LMCF that converges to a special Lagrangian, then $\gamma\in P^{1,1,\alpha}_{\mu}$, and $\theta_{\varepsilon}(\gamma)-\theta_0(\gamma)\in P^{0,0,\alpha}_{\mu}$. Moreover for $\delta>0$, if $\varepsilon$ is small enough then
\begin{align}
    \Vert \theta_{\varepsilon}(\gamma)-\theta_0(\gamma)\Vert_{P^{0,0,\alpha}_{\mu}}< \delta.
\end{align}
\end{prop}
\begin{proof}
    It follows from Proposition \ref{expconvergence} that $\gamma\in \wholdersp{1}{1}{}{}$. Then by Lemma \ref{lagestimate}, we know that $\theta_{\varepsilon}(\gamma)\in P^{0,0,\alpha}_{\mu}$. The smallness of $\Vert \theta_{\varepsilon}(\gamma)-\theta_0(\gamma)\Vert_{P^{0,0,\alpha}_{\mu}}$ follows from the convergence of the metrics $g_{\varepsilon}$ to $g_0$ as $\varepsilon\rightarrow 0$, because
    \begin{align}
        \theta_{\varepsilon}(\gamma)=-i\ln\big(\frac{\gamma^*\Omega_{\varepsilon}}{\text{Vol}_{\gamma_{\varepsilon}}}\big). 
    \end{align}
    
\end{proof}

Recall that 
\begin{align}\label{higherorder}
    Q_{\varepsilon}(v)= \theta_{\varepsilon}(\gamma+dv)-\theta_{\varepsilon}(\gamma)-\Delta_{\gamma_{\varepsilon}}v-\langle\nabla_{\gamma_{\varepsilon}}\theta_{\varepsilon}(\gamma),\hat{V}(dv)\rangle_{\gamma_{\varepsilon}}. 
\end{align}

To establish the higher order estimate we are going to use the following lemma, that is a direct application of Taylor's theorem with remainder. 

\begin{lemma}\label{taylorest}
    Let $f:V_1\times V_2\rightarrow W$ be a $C^{k+1}$ function between Banach spaces. Suppose that $D^j_{v_1}f\vert{(0,v_2)}=0$, for $j=0,\ldots, k$. Then there is $C>0$, and $\varepsilon>0$, such that if  $\Norm{v_2}<\varepsilon$, then
    \begin{align*}
        \Norm{f(v_1,v_2)}_{W}\leq C\Norm{v_1}_{V_1}^{k+1}. 
    \end{align*}
\end{lemma}

\begin{prop}[Quadratic and higher order estimate]\label{higherorderest}
    For small enough $\delta>0$, if $v\in B^{1,2,\alpha}_{\mu, \delta}$, then $Q_{\varepsilon}(v)\in P^{0,0,\alpha}_{\mu}$. Moreover, there exists $\varepsilon>0$ such that if $\varepsilon'<\varepsilon$, 
    \begin{align}
        \Vert Q_{\varepsilon'}(v)\Vert_{P^{0,0,\alpha}_{\mu}} \leq C\Vert v\Vert_{P^{1,2,\alpha}_{\mu}}^2.
    \end{align}
    for some $C>0$.   
\end{prop}
\begin{proof}
    It follows from Lemma \ref{lagestimate} that $Q_{\varepsilon}(v)\in \wholdersp{0}{0}{}{}$. The existence of the required constant follows from Lemma \ref{taylorest}.  
\end{proof}

\subsection{Applying Schauder fixed point Theorem}

We now put everything together to use the Schauder fixed point Theorem. 

We define an iteration map $\mathcal{F}: B^{1,2,\alpha}_{\mu,\delta}\rightarrow P^{1,2,\alpha}_{\mu}$ in the following way: to a given $v\in B^{1,2,\alpha}_{\delta,\mu}$ we associate $\mathcal{F}(v)=u\in P^{1,2,\alpha}_{\mu}$ such that $u$ is the unique solution of
\begin{align}
    \begin{cases}
        \der[u]{t}-\Delta_{\gamma_{\varepsilon}}u-\langle \nabla_{\gamma_{\varepsilon}}\theta_{\varepsilon}(\gamma),\hat{V}(du)\rangle_{\gamma_{\varepsilon}}=\theta_{\varepsilon}(\gamma)-\theta_0(\gamma)+Q_{\varepsilon}(v)-a(t),\\
        u(\cdot,0)=0,
    \end{cases}
\end{align}
where $a(t)=\int \theta_{\varepsilon}(\gamma)-\theta_0(\gamma)+Q_{\varepsilon}(v) \text{ Vol}_{\gamma_{\varepsilon}}$. Notice that we take $a(t)$ of this form so that we are able to use Theorem \ref{exttheory}. Moreover, we are using that $V_t=\nabla_{\gamma_{\varepsilon}}\theta_{\varepsilon}(\gamma)$, and $\hat{V}(du)$ satisfy the hypothesis of the theorem. This follows because $\gamma$ is a smooth LMCF that converges to a special Lagrangian, and because $\hat{V}(du)$ is the tangential component of $ \der[]{s}\vert_{s=0}\phi\circ(\gamma+ d(su))$.

We are going to use the following version of the Schauder fixed point Theorem (see Theorem 11.1 in \cite{gilbarg1977elliptic}), 

 \begin{theorem}[Schauder fixed point theorem]
     Let $K$ be a nonempty convex closed subset of a Banach space $V$. If $\mathcal{F}$ is a continuous map of $K$ into itself such that $\mathcal{F}(K)$ is contained in a compact subset of $K$, then $\mathcal{F}$ has a fixed point. 
 \end{theorem}

 In order to use the theorem to find a fixed point of $\mathcal{F}$, we need the following proposition

\begin{prop}
    For any $\mu'<\mu$, $\alpha'<\alpha$, the set $B^{1,2,\alpha}_{\mu,\delta}$ is compactly contained in $P^{1,2,\alpha'}_{\mu'}$. Moreover, for small enough $\delta>0$ the iteration map $\mathcal{F}:B^{1,2,\alpha}_{\mu, \delta}\rightarrow P^{1,2,\alpha}_{\mu}$ is continuous with respect to the norm on $P^{1,2,\alpha'}_{\mu'}$, and has image lying in $B^{1,2,\alpha}_{\mu,\delta}$. 
\end{prop}
\begin{proof}
    The first statement follows from our compactness result, Proposition \ref{Schaudercompactness}. That the image lies in $B^{1,2,\alpha}_{\mu,\delta}$ for small enough $\delta$, follows from the linear theory (Theorem \ref{exttheory}) and the estimates on the zeroth (Proposition \ref{zeroorderest}) and higher order terms (Proposition \ref{higherorderest}). The continuity of the map with respect  to the norm on $P^{1,2,\alpha'}_{\mu'}$ can be proven by a contradiction argument as in \cite{brendle2017gluing} (see Proposition 7.3). Suppose that $\mathcal{F}$ is not continuous with respect to this norm, then there is a sequence $u_j\in B_{\mu,\delta}^{1,2,\alpha}$ and $u\in B_{\mu,\delta}^{1,2,\alpha}$ such that
    \begin{align*}
        \lim_{j\rightarrow \infty}\Norm{u_j-u}_{P^{1,2,\alpha'}_{\mu'}}=0,
    \end{align*}
    but
    \begin{align*}
        \liminf_{j\rightarrow\infty}\Norm{\mathcal{F}(u_j)-\mathcal{F}(u)}_{P^{1,2,\alpha'}_{\mu'}}>0. 
    \end{align*}
    Since $F(u_j)\in B^{1,2,\alpha}_{\mu,\delta}$, then there is $v\in B^{1,2,\alpha}_{\mu,\delta}$ such that 
    \begin{align*}
        \limsup_{j\rightarrow \infty}\Norm{\mathcal{F}(u_j)-v}_{P^{1,2,\alpha'}_{\mu'}}=0.
    \end{align*}
    It follows that $v=\mathcal{F}(u)$, but this is impossible. 
\end{proof}

This allows us to use the Schauder fixed point theorem to conclude that, 

\begin{corollary}
    The map $\mathcal{F}:B^{1,2,\alpha}_{\mu,\delta}\rightarrow B^{1,2,\alpha}_{\mu,\delta}$ has a fixed point.
\end{corollary}

\begin{remark}
    Up until now we have proven that there exists a solution to the equation in $P^{1,2,\alpha}_{\mu}$. It follows from parabolic regularity theory that this solution is actually smooth. 
\end{remark}

\section{LMCF in the Kummer K3}

To conclude this paper, we explain how to use our perturbation result to construct examples of LMCF in the Kummer K3 surface, by perturbing the solutions found by Lotay and Oliveira in the Eguchi--Hanson space. 

The Kummer K3 surface is the smooth complex surface that results from blowing-up the 16 singular points of $T^4/\{\pm 1\}$. Notice that each of the singular points is modeled on the zero point of $\mathbb{R}^4/\{\pm 1\}$, whose desingularization is $T^*\mathbb{S}^2$. Because of this, the blow-up procedure is equivalent to removing a neighborhood of each singular point, and gluing in a neighborhood of the zero section in $T^*\mathbb{S}^2$. 

Through a gluing construction, using the flat metric in $T^4/\{\pm 1\}$ and the Eguchi--Hanson metric in $T^*\mathbb{S}^2$, one can endow the Kummer K3 surface with a one parameter family of Ricci-flat metrics.  This parameter can be thought of as the inverse of the length of a cylindrical region between the Eguchi--Hanson and the torus parts. The idea of the construction is as follows, using a cut-off function one constructs a one-parameter family of K\"ahler metrics $\omega_r$ that are almost Ricci-flat. In the regions modeled on $T^*\mathbb{S}^2$, $\omega_r$ is the Eguchi--Hanson metric, and in the torus region it is the flat metric. Afterwards, one solves the complex Monge--Amp\`ere equation
\begin{align*}
    (\omega_r+i\partial\overline{\partial}\varphi_r)^2=\lambda \Omega\wedge\overline{\Omega}, 
\end{align*}
for a fixed constant $\lambda$, and for $r$ sufficiently small. See Donaldson's paper \cite{donaldson2010calabi} for more details. 

The main takeaway is that one obtains a family of Calabi--Yau structures in the Kummer K3 surface, where the metric, and symplectic structures are changing, but the complex structure is fixed. In the regions modeled on $T^*\mathbb{S}^2$, the zero section is a complex submanifold, in particular $\Omega\vert_{\mathbb{S}^2}=0$.

The Gibbons--Hawking ansatz describes all hyperk\"ahler 4-manifolds admitting a tri-Hamiltonian circle action, this includes the Eguchi--Hanson metrics. Lotay and Oliveira proved the following version of the Thomas--Yau conjecture \cite{lotay2020special} for circle invariant Lagrangians in one of the manifolds arising from the Gibbons--Hawking ansatz. 

\begin{theorem}[See Theorem 1.2 in \cite{lotay2020special}]
    Let $X^4$ be an ALE or ALF hyperk\"ahler 4-manifold arising from the Gibbons--Hawking ansatz and let $L^2\subset X^4$ be a compact, embedded, almost calibrated, circle-invariant Lagrangian. If $L$ is flow stable, then the Lagrangian mean curvature flow starting at $L$ exists for all time and converges smoothly to the unique circle invariant special Lagrangian in its Hamiltonian isotopy class. 
\end{theorem}

In the case of the Eguchi--Hanson metric, the flow stability condition is always satisfied (see Definition 6.2, and Lemma 6.8 in \cite{lotay2020special}), therefore:

\begin{corollary}\label{LMCF in EH}
    Let $T^*\mathbb{S}^2$ be endowed with the Eguchi--Hanson metric, and choose a compatible Calabi--Yau structure on $T^*\mathbb{S}^2$ so that the zero section is special Lagrangian. Let $L\subset T^*\mathbb{S}^2$ be a compact, embedded, zero Maslov, circle-invariant Lagrangian. If $L$ is almost calibrated then the Lagrangian mean curvature flow starting at $L$ exists for all time and converges to $\mathbb{S}^2$. 
\end{corollary}

We can use our perturbation result to study LMCF in the regions of the Kummer K3 surface that are modeled on the Eguchi--Hanson space.

\begin{theorem}\label{Maintheorem1Tech}
    Consider the Kummer K3 surface $(K,J)$. There is a one parameter family of Calabi--Yau structures $(K,J,\omega_{\varepsilon},\Omega)$. Moreover, there exists an open neighborhood $U\subset T^*{\mathbb{S}^2}$ of the zero section, and sixteen neighborhoods $U_i\subset K$ such that 
    \begin{align*}
        (U_i, \varepsilon^{-2}\omega_{\varepsilon},\Omega)\rightarrow (U,\omega_{EH},\Omega), \text{ as } \varepsilon\rightarrow 0.
    \end{align*}
    With respect to the structures induced in $U_i$ by $\varepsilon^{-2}\omega_{\varepsilon}$, the zero section is a complex submanifold. 
    
    After a hyperk\"ahler rotation, in $(U, Im\Omega, \omega_{EH}+iRe\Omega)$ there exist circle invariant LMCF that converge to the zero section. If the initial condition is graphical over the zero section, then for small enough $\varepsilon>0$ the LMCF can be perturbed to a LMCF in $(U, Im\Omega,\varepsilon^{-2}\omega_{\varepsilon}+iRe\Omega)$, that also converges to the zero section.
\end{theorem}

\begin{proof}

As mentioned above, the existence of the Calabi--Yau structures can be proven by a gluing construction as in \cite{donaldson2010calabi}. The convergence as $\varepsilon\rightarrow 0$ of the rescaled structures is proven in \cite{foscolo2019alf}. Notice that the zero section is a complex submanifold of the Eguchi--Hanson space. Since the gluing construction preserves the complex structure, the copy of the zero section is also a complex submanifold in $(K,J,\omega_{\varepsilon},\Omega)$ for every $\varepsilon>0$. After the hyperk\"ahler rotation it becomes a special Lagrangian. 

By Corollary \ref{LMCF in EH} we know that an LMCF in the Eguchi--Hanson space that starts at an almost-calibrated, circle invariant Lagrangian, will converge to the zero section of $T^*\mathbb{S}^2$. 

If the initial condition of the previously mentioned LMCF is graphical over the zero section, then we can use our perturbation result, Theorem \ref{deformationtheo}, to obtain an LMCF in $(U, Im\Omega,\varepsilon^{-2}\omega_{\varepsilon}+iRe\Omega)$, that converges to the zero section. 
\end{proof}

\bibliographystyle{acm}
\bibliography{bibliografia}

\end{document}